\documentclass[12pt, leqno]{amsart}
\usepackage{amsfonts, amssymb, amsmath, amscd, lastpage, fancyheadings, array}
\usepackage[pdftex]{graphicx}

\newtheorem{thm}{Theorem}
\newtheorem{theorem}{Theorem}[section]

\newtheorem{lem}[theorem]{Lemma}
\newtheorem{conj}[theorem]{Conjecture}
\newtheorem{lemma}[theorem]{Lemma}

\newtheorem{coro}[theorem]{Corollary}

\newcommand{\al}{\alpha}
\newcommand{\be}{\beta}
\newcommand{\del}{\delta}
\newcommand{\ep}{\epsilon}

\newcommand{\ka}{\kappa}
\newcommand{\la}{\lambda}
\newcommand{\om}{\omega}

\newcommand{\Z}{\mbox{$\mathbb Z$}}

\newcommand{\R}{\mbox{$\mathbb R$}}     

\begin{document}
\title[Baker's Explicit abc-Conjecture and applications]{Baker's Explicit abc-Conjecture and applications}
\author{Shanta Laishram}
\email{shanta@isid.ac.in}
\address{Stat Math Unit, Indian Statistical Institute,
7 SJS Sansanwal Marg, New Delhi 110016, India}
\author[T. N. Shorey]{T. N. Shorey}
\email{shorey@math.iitb.ac.in}
\address{Department of Mathematics\\
Indian Institute of Technology Bombay, Powai, Mumbai 400076, India}
\keywords{ABC Conjecture, Generalized Fermat Equation}
\dedicatory{Dedicated to Professor Andrzej Schinzel on his 70th Birthday}
\begin{abstract}
The conjecture of Masser-Oesterl\'e, popularly known as $abc$-conjecture have
many consequences. We use an explicit version due to Baker to solve
a number of conjectures.
\end{abstract}
\maketitle

\section{Introduction}

The well known conjecture of Masser-Oesterle states that

\begin{conj}\label{abc} {\bf Oesterl\'e and Masser's abc-conjecture:}
For any given $\epsilon>0$ there exists a computable constant
$\frak{c}_{\epsilon}$ depending only on $\epsilon$ such that if
\begin{align}\label{a+b=c}
a+b=c
\end{align}
where $a, b$ and $c$ are coprime positive integers, then
\begin{align*}
c\le \frak{c}_{\epsilon}\left(\prod_{p|abc}p\right)^{1+\epsilon}.
\end{align*}
\end{conj}

It is known as $abc$-conjecture; the name derives from the usage of letters
$a, b, c$ in \eqref{a+b=c}. For any positive integer $i>1$, let $N=N(i)=\prod_{p|i}p$ be
the radical of $i$, $P(i)$ be the greatest prime factor of $i$ and $\om(i)$ be the
number of distinct prime factors of $i$ and we put $N(1)=1, P(1)=1$ and $\om(1)=0$. An
explicit version of this conjecture due to Baker \cite{baker} is the following:

\begin{conj}\label{Baker} {\bf Explicit abc-conjecture:}
Let $a, b$ and $c$ be pairwise coprime positive integers satisfying \eqref{a+b=c}. Then
\begin{align*}
c<\frac{6}{5}N\frac{(\log N)^{\om}}{\om !}
\end{align*}
where $N=N(abc)$ and $\om=\om(N)$.
\end{conj}

We observe that $N=N(abc)\ge 2$ whenever $a, b, c$ satisfy \eqref{a+b=c}. We shall
refer to Conjecture \ref{abc} as \emph{$abc-$conjecture} and Conjecture \ref{Baker} as
\emph{explicit $abc-$conjecture}.
Conjecture \ref{Baker} implies the following explicit version of Conjecture \ref{abc}.

\begin{thm}\label{abcexplicit}
Assume Conjecture \ref{Baker}. Let $a, b$ and $c$ be pairwise coprime positive integers
satisfying \eqref{a+b=c} and $N=N(abc)$. Then we have
\begin{align}\label{3/4}
c<N^{1+\frac{3}{4}}.
\end{align}
Further for $0<\epsilon\le \frac{3}{4}$, there exists $\om_\ep$ depending only $\ep$ such that
when $N=N(abc)\ge N_\ep=\prod_{p\le p_{\om_\ep}}p$, we have
\begin{align*}
c< \ka_\ep N^{1+\epsilon}
\end{align*}
where
\begin{align*}
\ka_\ep=\frac{6}{5\sqrt{2\pi \max(\om, \om_\ep)}}\le \frac{6}{5\sqrt{2\pi \om_\ep}}
\end{align*}
with $\om=\om(N)$. Here are some values of $\ep, \om_\ep$ and $N_\ep$.
\begin{center}
\begin{tabular}{|c|c|c|c||c|c|c|c|} \hline
$\ep$ & $\frac{3}{4}$ & $\frac{7}{12}$ & $\frac{6}{11}$ & $\frac{1}{2}$ & $\frac{34}{71}$ &
$\frac{5}{12}$ & $\frac{1}{3}$ \\ \hline
$\om_\ep$ & $14$ & $49$ & $72$ & $127$ & $175$ & $548$ & $6460$ \\ \hline
$N_\ep$ & $e^{37.1101}$ & $e^{204.75}$ & $e^{335.71}$ & $e^{679.585}$ & $e^{1004.763}$
& $e^{3894.57}$ & $e^{63727}$ \\ \hline
\end{tabular}
\end{center}
\end{thm}
Thus $c<N^2$ which was conjectured in Granville and Tucker \cite{gran}.
As a consequence of Theorem \ref{abcexplicit}, we have

\begin{thm}\label{ABCErdos}
Assume Conjecture \ref{Baker}. Then the equation
\begin{align}\label{byl}
n(n+d)\cdots (n+(k-1)d)=by^{\ell}
\end{align}
in integers $n\geq 1, d>1, k\ge 4, b\geq 1, y\geq 1, \ell>1$ with gcd$(n, d)=1$ and $P(b)\leq k$ implies
$\ell \leq 7$. Further $k<e^{13006.2}$ when $\ell=7$.
\end{thm}

We observe that $e^{13006.2}<e^{e^{9.52}}$. Assuming $abc-$conjecture, Shorey \cite{bam}
proved that \eqref{byl} with $\ell \ge 4$ implies that $k$ is bounded by an absolute constant,
the assertion for $\ell \in \{2, 3\}$ is due to Granville (see Laishram \cite[p. 69]{Mthesis}).
For a given $k\ge 3$, Gy\H{o}ry, Hajdu and Saradha \cite{GyHaSa} showed that $abc-$conjecture
implies that \eqref{byl} has only finitely many solutions in positive integers $n, d>1, b, y$
and $\ell \ge 4$. Saradha \cite{abcSara} showed that \eqref{byl} with $k\geq 8$ implies that
$\ell \leq 29$ and further $k\leq 8, 32, 10^2, 10^7$ and $e^{e^{280}}$ according as
$\ell=29$, $\ell \in \{23, 19\}, \ell=17, 13$ and $\ell \in \{11, 7\}$, respectively.
It has been conjectured that
$(k,l) \in \{(3,3),(4,2),(3,2)\}$ whenever there are positive integers $n, d>1, y\ge 1, b,
\ell\ge 2$ and $k\ge 3$ with gcd$(n,d)=1$ and $P(b)\le k$ satisfying \eqref{byl} and it is
known that \eqref{byl} has infinitely many solutions when $(k, \ell)\in \{(3, 2), (3,3) (4,2)\}$.
For an account of results on \eqref{byl}, we refer to Shorey \cite{rims}, \cite{Camb} and
Shorey and Saradha \cite{Contr}.

Nagell-Ljunggren equation is the equation
\begin{align}\label{nalu}
y^q=\frac{x^n-1}{x-1}
\end{align}
in integers $x>1, y>1, n>2, q>1$. It is known that
\begin{align*}
11^2=\frac{3^5-1}{3-1}, 20^2=\frac{7^4-1}{7-1},  7^3=\frac{18^3-1}{18-1}
\end{align*}
which are called the \emph{exceptional solutions}. Any other solution is
termed as \emph{non-exceptional solutions}. For an account of results on
\eqref{nalu}, see Shorey \cite{bam} and Bugeaud and Mignotte \cite{bumi}.
It is conjectured that there are no \emph{non-exceptional solutions}. We prove
in Section 7 the following.

\begin{thm}\label{ABCnalu}
Assume Conjecture \ref{Baker}. There are no non-exceptional solutions of
equation \eqref{nalu} in integers $x>1, y>1, n>2, q>1$.
\end{thm}

Let $(p, q, r)\in \Z_{\ge 2}$ with $(p, q, r)\neq (2, 2, 2)$. The equation
\begin{align}\label{f-c-e}
x^p+y^q=z^r, \ \ (x, y, z)=1, x, y, z\in \Z
\end{align}
is called the \emph{Generalized Fermat Equation} or \emph{Fermat-Catalan Equation} with
signature $(p, q, r)$. An integer solution $(x, y, z)$ is said to be non-trivial if
$xyz\neq 0$ and primitive if $x, y, z$ are coprime.
We are interested in finding non-trivial primitive integer solutions of \eqref{f-c-e}.
The case $p=q=r$ is the famous \emph{Fermat's equation} which is completely solved by
Wiles \cite{wiles}. One of known solution $1^p+2^3=3^2$ of \eqref{f-c-e} comes from
\emph{Catalan's equation}. Let $\chi=\frac{1}{p}+\frac{1}{q}+\frac{1}{r}-1$.
The parametrization of nontrivial primitive integer solutions for
$(p, q, r)$ with $\chi \ge 0$ is completely solved (\cite{beu}, \cite{cohen}). It was shown by
Darmon and Granville \cite{dargra} that \eqref{f-c-e} has only finitely many
equations in $x, y, z$ if $\chi<0$.When $2\in  \{p, q, r\}$, there are some
known solutions. So, we consider $p\ge 3, q\ge 3, r\ge 3$. An open problem in
this direction is the following.
\begin{conj}{\bf Tijdeman, Zagier:}
There are no non-trivial solutions to \eqref{f-c-e} in positive integers $x, y, z, p, q, r$
with $p\ge 3, q\ge 3$ and $r\ge 3$.
\end{conj}
This is also referred to as \emph{Beal's Conjecture} or \emph{Fermat-Catalan Conjecture}.
This conjecture has been established for many signatures $(p, q, r)$,
including for several infinite families of signatures. For exhaustive surveys,
see \cite{beu}, \cite[Chapter 14]{cohen}, \cite{kraus} and \cite{poo}. Let $[p, q, r]$
denote all permutations of ordered triples $(p, q, r)$ and let
\begin{align*}
Q=\{[3, 5, p]: 7\le p\le 23, p \ {\rm prime}\}\cup \{[3, 4, p]: p \ {\rm prime}\}.
\end{align*}
We prove the following in Section 8.

\begin{thm}\label{ABCf-c-e}
Assume Conjecture \ref{Baker}. There are no non-trivial solutions
to \eqref{f-c-e} in positive integers $x, y, z, p, q, r$ with
$p\ge 3, q\ge 3$ and $r\ge 3$ with $(p, q, r)\not\in Q$. Further
for $(p, q, r)\in Q$, we have $\max (x^p, y^q, z^r)<e^{1758.3353}$.
\end{thm}

Another equation which we will be considering is the equation of Goormaghtigh
\begin{align}\label{goor}
\frac{x^m-1}{x-1}=\frac{y^n-1}{y-1} \ {\rm integers} \ x>1, y>1, m>2, n>2 \
{\rm with} \ x\neq y.
\end{align}
We may assume without loss of generality that $x>y>1$ and $2<m<n$. It is known that
\begin{align}\label{goorsol}
31=\frac{5^3-1}{5-1}=\frac{2^5-1}{2-1} \ {\rm and} \
8191=\frac{90^3-1}{90-1}=\frac{2^{13}-1}{2-1}
\end{align}
are the solutions of \eqref{goor} and it is conjectured that there are no other
solutions. A weaker conjecture states that there are only finitely many solutions
$x, y, m, n$ of \eqref{goor}. We refer to \cite{bam} for a survey of results on \eqref{goor}.
We prove in Section 9 that
\begin{thm}\label{ABCgoor}
Assume Conjecture \ref{Baker}. Then equation \eqref{goor} in integers $x>1, y>1, m>2, n>3$
with $x>y$ implies that $m\leq 6$ and further $7\le n\le 17, n\notin \{11, 16\}$ if $m=6$; moreover there exists
an effectively computable absolute constant C such that
\begin{align*}
\max(x,y,n)\leq C.
\end{align*}
\end{thm}

Thus, assuming Conjecture \ref{Baker},  equation \eqref{goor} has only finitely many solutions
in integers $x>1, y>1, m>2, n>3$ with $x\neq y$ and this improves considerably
Saradha \cite[Theorem  1.4]{abcSara}.

\section{Notation and Preliminaries}

For an integer $i>0$, let $p_i$ denote the $i-$th prime. For a real $x>0$,
let $\Theta(x)=\prod_{p\le x}p$ and $\theta(x)=\log (\Theta(x))$. We write
$\log_2 i$ for $\log(\log i)$. We have

\begin{lemma}\label{pnt}
We have
\begin{enumerate}
\item[$(i)$]{$\displaystyle{\pi(x)\le \frac{x}{\log x}\left(1+\frac{1.2762}{\log x}\right) \ {\rm for} \ x>1.}$}
\item[$(ii)$]{$\displaystyle{p_i \ge i(\log i+\log_2i-1) \ {\rm for} \ i\ge 1}$}
\item[$(iii)$]{$\displaystyle{\theta(p_i)\ge i(\log i+\log_2i-1.076869) \ {\rm for} \ i\ge 1}$}
\item[$(iv)$]{$\displaystyle{\theta(x)<1.000081x \ {\rm for} \ x>0}$}
\item[$(v)$] ${\rm ord}_p(k!)\geq \frac{k-p}{p-1}-\frac{\log (k-1)}{\log p} \ {\rm for} \ ~p<k.$
\item[$(vi)$] $\sqrt{2\pi k}(\frac{k}{e})^ke^{\frac{1}{12k+1}}\leq k!\leq \sqrt{2\pi k}(\frac{k}{e})^ke^{\frac{1}{12k}}.$
\end{enumerate}
\end{lemma}
Here we understand that $\log_2 1=-\infty$. The estimates $(i)$ and $(ii)$ are due to
Dusart, see \cite{dus1} and \cite{duspk}, respectively.
The estimate $(iii)$ is \cite[Theorem 6]{robin}. For estimate $(iv)$, see \cite{dus1}.
For a proof of $(v)$, see \cite[Lemma 2(i)]{shanta2}. The estimate $(vi)$ is \cite[Theorem 6]{robb}.

\section{Proof of Theorem \ref{abcexplicit}}

Let $\ep>0$ and $N\ge 1$ be an integer with $\om(N)=\om$. Then $N\geq \Theta(p_\om)$
or $\log N\geq \theta(p_\om)$. Given $\om$, we observe that $\frac{M^\ep}{(\log M)^\om}$
is an increasing function for $\log M\geq \frac{\om}{\ep}$.
Let
\begin{align*}
X_0(i)=\log i+\log_2 i -1.076869.
\end{align*}
Then $\theta(p_\om)\ge \om X_0(\om)$ by Lemma \ref{pnt} $(iii)$. Observe
that $X_0(i)>1$ for $i\ge 5$. Let $\om_1\ge 5$ be smallest $\om $ such that
\begin{align}\label{omep}
\ep X_0(\om)-\log X_0(\om)\geq 1 \ \ {\rm for \ all} \  \om\geq \om_1.
\end{align}
Note that $\ep X_0(\om)\ge 1$ for $\om \ge \om_1$ implying
$\log N\ge \theta(p_\om)\ge \om X_0(\om)\ge \frac{\om}{\ep}$
for $\om \ge \om_1$ by Lemma \ref{pnt} $(iii)$. Therefore
\begin{align*}
\frac{\om!N^\ep}{(\log N)^\om}\geq \frac{\om! \Theta(p_{\om})^\ep}{(\theta(p_{\om}))^{\om}}
\geq  \frac{\om!e^{\ep \om X_0(\om)}}{(\om X_0(\om))^{\om}}
>\sqrt{2\pi \om}(\frac{\om}{e})^{\om}\frac{e^{\ep \om X_0(\om)}}{(\om X_0(\om))^{\om}} \
{\rm for} \ \om \ge \om_1.
\end{align*}
Thus for $\om\ge \om_1$, we have from \eqref{omep} that
\begin{align*}
\log\left(\frac{\om!e^{\ep \om X_0(\om)}}{(\om X_0(\om))^{\om}}\right)>
&\log \sqrt{2\pi \om}+\om(\log(\om)-1)+\ep \om X_0(\om)-\om (\log \om +\log X_0(\om))\\
>&\log \sqrt{2\pi \om}+\om(\ep X_0(\om)-\log X_0(\om)-1)\geq \log \sqrt{2\pi \om}
\end{align*}
implying
\begin{align*}
\frac{\om!N^\ep}{(\log N)^\om}\ge \frac{\om! \Theta(p_{\om})^\ep}{(\theta(p_{\om}))^{\om}}
> \sqrt{2\pi \om}  \ {\rm for} \ \om \ge \om_1.
\end{align*}
Define $\om_\ep$ be the smallest $\om\le \om_1$ such that
\begin{align}\label{omepe}
\theta(p_\om)\ge \frac{\om}{\ep} \ {\rm and} \
\frac{\om! \Theta(p_{\om})^\ep}{(\theta(p_{\om}))^{\om}}>\sqrt{2\pi \om}
\ {\rm for \ all} \ \om_\ep \leq \om \le \om_1
\end{align}
by taking the exact values of $\om$ and $\theta$. Then clearly
\begin{align}\label{inte}
\frac{\om!N^\ep}{(\log N)^\om}\ge
\frac{\om! \Theta(p_{\om})^\ep}{(\theta(p_{\om}))^{\om}}
> \sqrt{2\pi \om}  \ {\rm for} \ \om \ge \om_\ep.
\end{align}
Here are values of $\om_\ep$ for some $\ep$ values.
\begin{center}
\begin{tabular}{|c|c|c|c|c|c|c|c|} \hline
$\ep$ & $\frac{3}{4}$ & $\frac{7}{12}$ & $\frac{6}{11}$ & $\frac{1}{2}$ &
$\frac{34}{71}$ & $\frac{5}{12}$ & $\frac{1}{3}$ \\ \hline
$\om_\ep $ & $14$ & $49$ & $72$ & $127$ & $175$ & $548$ & $6458$ \\ \hline
\end{tabular}
\end{center}
Let $\om<\om_\ep$ and $N\ge \Theta(\om_\ep)$. Then
$\log N\ge \theta(\om_\ep)\ge \frac{\om_\ep}{\ep}$. Therefore
\begin{align*}
\frac{\om! N^\ep}{(\log N)^\om}\geq \frac{\om!
\Theta(p_{\om_\ep})^\ep}{(\theta(p_{\om_\ep}))^{\om}}
=\frac{\om_\ep!\Theta(p_{\om_\ep})^\ep}{(\theta(p_{\om_\ep}))^{\om_\ep}}\cdot
\frac{\om!}{\om_\ep!}(\theta(p_{\om_\ep}))^{\om_\ep-\om}> \sqrt{2\pi \om_\ep}
\frac{\om! \om^{\om_\ep-\om}_\ep}{\om_\ep!}\ge \sqrt{2\pi \om_\ep}.
\end{align*}
Combining this with \eqref{inte}, we obtain
\begin{align}\label{omep2pi}
\frac{(\log N)^{\om}}{\om !}<\frac{N^\ep}{\sqrt{2\pi \max(\om, \om_\ep)}}\le
\frac{N^\ep}{\sqrt{2\pi \om_\ep}} \ {\rm for} \  N\geq \Theta(\om_\ep).
\end{align}
Further we now prove
\begin{align}\label{omep65}
\frac{(\log N)^{\om}}{\om !}<\frac{5N^{\frac{3}{4}}}{6} \ {\rm for} \  N\ge 1.
\end{align}
For that we take $\ep=\frac{3}{4}$. Then $\om_\ep=14$ and we may assume that
$N<\Theta(p_{14})$. Then $\om=\om(N)<14$. Observe that $N\ge \Theta(p_\om)$ and
$\frac{N^{\frac{3}{4}}}{(\log N)^\om}$ is increasing for $\log N\geq \frac{4\om}{3}$.
For $4\le \om<14$, we check that
\begin{align*}
\theta(p_\om)\ge \frac{4\om}{3} \ {\rm and} \
\frac{\om! \Theta(p_{\om})^{\frac{3}{4}}}{(\theta(p_{\om}))^{\om}}> \frac{6}{5}
\end{align*}
implying \eqref{omep65} when $4\le \om=\om(N)<14$. Thus we may assume that $\om=\om(N)<4$.
We check that
\begin{align}\label{tempo65}
\frac{\om! N^{\frac{3}{4}}}{(\log N)^{\om}}> \frac{6}{5} \ {\rm at} \ N=e^\frac{4\om}{3}
\end{align}
for $1\leq \om<4$ implying \eqref{omep65} for $N\ge e^\frac{4\om}{3}$. Thus we may assume
that $N<e^\frac{4\om}{3}$. Then $N\in \{2, 3\}$ if $\om=\om(N)=1$, $N\in \{6, 10, 12, 14\}$
if $\om=\om(N)=2$ and
$N\in \{30, 42\}$ if $\om(N)=3$. For these values of $N$ too, we find that \eqref{tempo65} is
valid implying \eqref{omep65}. Clearly \eqref{omep65} is valid when $N=1$.

We now prove Theorem \ref{abcexplicit}.
Assume Conjecture \ref{Baker}. Let $\ep>0$ be given. Let
$a, b, c$ be positive integers such that $a+b=c$ and gcd$(a, b)=1$.
By Conjecture \ref{Baker}, $c\le \frac{6}{5}N\frac{(\log N)^\om}{\om!}$ where $N=N(abc)$.
Now assertion \ref{3/4} follows from \eqref{omep65}. Let $0<\ep \leq \frac{3}{4}$
and $N_\ep=\Theta(p_{\om_\ep})$. By \eqref{omep2pi}, we have
\begin{align*}
c<\frac{6N^{1+\ep}}{5\sqrt{2\pi \max(\om, \om_\ep)}}.
\end{align*}
The table is obtained by taking the table values of $\ep, \om_\ep$ given after \eqref{inte}
and computing $N_\ep$ for those $\ep $ given in the table. Hence the Theorem. $\hfill \qed$

\section{Proof of Theorem \ref{ABCErdos}}

Let $n, d, k, b, y$ be positive integers with $n\geq 1, d>1, k\ge 4, b\ge 1, y\ge 1$, gcd$(n, d)=1$ 
and $P(b)\leq k$. We consider the Diophantine equation
\begin{align}\label{byl1}
n(n+d)\cdots (n+(k-1)d)=by^{\ell}.
\end{align}
Observe that $P(n(n+d)\cdots (n+(k-1)d))>k$ by a result of Shorey and Tijdeman \cite{shoti}
and hence $P(y)>k$ and $n+(k-1)d>(k+1)^\ell$. For every $0\le i<k$, we write
\begin{align*}
n+id=A_iX^\ell_i \ {\rm with} \ P(A_i)\leq k \ {\rm and} \ (X_i, \prod_{p\le k}p)=1.
\end{align*}
Without loss of generality, we may assume that $k=4$ or $k\geq 5$ is a prime which we assume throughout in this
section. We observe that $(A_i, d)=1$ for $0\le i<k$ and $(X_i, X_j)=1$. Let
\begin{align*}
S_0=\{A_0, A_1, \ldots, A_{k-1}\}.
\end{align*}
For every prime $p\le k$ and $p\nmid d$, choose $i_p$ be such that
ord$_p(A_i)=$ord$_p(n+id)\leq$ord$_p(n+i_pd)$ for $0\leq i<k$.
For a $S\subset S_0$, let
\begin{align*}
S'=S-\{A_{i_p}: p\leq k, p\nmid d\}.
\end{align*}
Then $|S'|\ge |S|-\pi_d(k)$. By Sylvester-Erd\H{o}s inequality(see
\cite[Lemma 2]{ersel} for example), we obtain
\begin{align}\label{erd}
\prod_{A_i\in S'}A_i|(k-1)!\prod_{p|d}p^{-{\rm ord}_p((k-1)!)}.
\end{align}
As a consequence, we have

\begin{lemma}\label{S1<k/4}
Let $\al, \be\in \R$ with $\al\ge 1$ and $e\be<\al$. Let
\begin{align*}
S_1:=S_1(\al):=\{A_i\in S_0: A_i\leq \al k\}.
\end{align*}
For
\begin{align}\label{kbe}
k\geq \frac{\log(\frac{e\al}{\sqrt{\be}})+
\frac{k\log(\al k)}{\log k}\left(1+\frac{1.2762}{\log k}\right)-\log(\al k)}
{\log (e\al) +\be \log \left(\frac{\be}{e\al}\right)},
\end{align}
we have $|S_1|>\be k$.
\end{lemma}

\begin{proof}
Let $S=S_0$, $s_1=|S_1|$ and $s_2=|S'-S_1|$. Then $s_2\ge k-\pi(k)-s_1$. We get
from \eqref{erd} that
\begin{align}\label{k/9}
s_1!\prod^{k-\pi(k)-s_1}_{i=1}([\al k+i])\le \prod_{A_i\in S'}A_i \leq (k-1)!
\end{align}
since elements of $S'-S_1$ are distinct. Using Lemma \ref{pnt} $(vi)$, we obtain
\begin{align*}
(\al k)^{k-\pi(k)} < \frac{(k-1)!}{s_1!}(\al k)^{s_1}<
\begin{cases}
\sqrt{2\pi(k-1)}\left(\frac{k-1}{e}\right)^{k-1}e^{\frac{1}{12(k-1)}} \
&{\rm if} \ s_1=0\\
\sqrt{\frac{k-1}{s_1}}
\left(\frac{\al ke}{s_1}\right)^{s_1}\left(\frac{k-1}{e}\right)^{k-1} \
&{\rm if} \ s_1>0.
\end{cases}
\end{align*}
We check that the expression for $s_1=0$ is less than that of $s_1=1$ since $\al\ge 1$.
Suppose $s_1\leq \be k$. Observe that
\begin{align*}
\sqrt{\frac{k-1}{s_1}}(\frac{\al k e}{s_1})^{s_1}
\end{align*}
is an increasing function of $s_1$ since $s_1\le \be k$ and $e\be<\al $. This can be
verified by taking $\log$ of the above expression and differentiating it with respect
to $s_1$. Therefore
\begin{align*}
(\al k)^{k-\pi(k)}<\sqrt{\frac{k-1}{\be k}}\left(\frac{e\al }{\be}\right)^{\be k}\left(\frac{k-1}{e}\right)^{k-1}
<\sqrt{\frac{1}{\be}}\left(\frac{e\al }{\be}\right)^{\be k}\left(\frac{k}{e}\right)^{k-1}
\end{align*}
implying
\begin{align*}
(e\al)^{k}\left(\frac{\be}{e\al}\right)^{\be k}<\frac{e\al}{\sqrt{\be}}(\al k)^{\pi(k)-1}.
\end{align*}
Using Lemma \ref{pnt} $(i)$, we obtain
\begin{align*}
\log (e\al) +\be \log \left(\frac{\be}{e\al}\right)<\frac{1}{k}\log(\frac{e\al}{\sqrt{\be}})+
\frac{\log(\al k)}{\log k}\left(1+\frac{1.2762}{\log k}\right)-\frac{\log(\al k)}{k}.
\end{align*}
The right hand side of the above inequality is a decreasing function of $k$ for $k$ given
by \eqref{kbe}. This can be verified by observing that $\frac{\log \al k}{\log k}=
1+\frac{\log \al}{\log k}$ and differentiating $\frac{1.2762+\log \al}{\log k}-\frac{\log(\al k)}{k}$
with respect to $k$. This is a contradiction for $k$ given by \eqref{kbe}.
\end{proof}

\begin{coro}\label{3<4k}
For $k>113$, there exist $0\leq f<g<h<k$ with $h-f\leq 8$ such that max$(A_f, A_g, A_h)\leq 4k$.
\end{coro}

\begin{proof}
By dividing $[0, k-1]$ into subintervals of the form $[9i, 9(i+1))$, it suffices to show
$S_1(4)>2([\frac{k}{9}]+1)$ where $S_1$ is as defined in Lemma \ref{S1<k/4}. Taking
$\al=4, \be=\frac{1}{4}$, we obtain from Lemma \ref{S1<k/4} that for $k\geq 700$,
$|S_1(4)|>\frac{k}{4}>2([\frac{k}{9}]+1)$. Thus we may suppose $k<700$ and
$|S_1(4)|\leq 2([\frac{k}{9}]+1)$. For each prime $k$ with $113<k<700$, taking $\al=4$
and $\be k=2([\frac{k}{9}]+1)$ in Lemma \ref{S1<k/4}, we get a contradiction from \eqref{k/9}.
Therefore $|S_1(4)|>2([\frac{k}{9}]+1)$ and the assertion follows.
\end{proof}

Given $0\leq f<g<h\leq k-1$, we have
\begin{align}\label{fgh}
(h-f)A_gX^\ell_g=(h-g)A_fX^\ell_f+(g-f)A_hX^\ell_h.
\end{align}
Let $\la=$gcd$(h-f, h-g, g-f)$ and write $h-f=\la w, h-g=\la u, g-f=\la v$.
Rewriting $h-f=h-g+g-f$ as
$$w=u+v \ {\rm with \ gcd}(u, v)=1,$$
\eqref{fgh} can be written as
\begin{align}\label{uvw}
wA_gX^\ell_g=uA_fX^\ell_f+vA_hX^\ell_h.
\end{align}
Let $G={\rm gcd}(wA_g, uA_f, vA_h)$,
\begin{align}\label{G}
r=\frac{uA_f}{G}, s=\frac{vA_h}{G}, t=\frac{wA_g}{G}
\end{align}
and we rewrite \eqref{uvw} as
\begin{align}\label{rst}
tX^\ell_g=rX^\ell_f+sX^\ell_h.
\end{align}
Note that gcd$(rX^\ell_f, sX^\ell_h)=1$.

From now on, we assume explicit $abc-$conjecture. Given $\ep>0$, let
$N(rstX_fX_gX_h)\ge N_\ep$ which we assume from now on till
the expression \eqref{Xg<3}. By Theorem \ref{abcexplicit}, we obtain
\begin{align}\label{abc0}
tX^\ell_g<\ka_\ep N(rstX_fX_gX_h)^{1+\ep}
\end{align}
i.e.,
\begin{align}\label{abc1}
X^\ell_g<\ka_\ep \frac{N(rst)^{1+\ep}(X_fX_gX_h)^{1+\ep}}{t}.
\end{align}
Here $N_\ep=\ka_\ep=1$ if $\ep\ge \frac{3}{4}$ and we may also
take $\ka_{\frac{3}{4}}\le \frac{6}{5\sqrt{28\pi}}$ if $N(rstX_fX_gX_h)\ge N_{\frac{3}{4}}$.
We will be taking $\ep=\frac{3}{4}$ for $\ell>7$ and $\ep \in\{\frac{5}{12},
\frac{1}{3}\}$ for $\ell=7$. We have from \eqref{abc0} that
\begin{align*}
rst(X_fX_gX_h)^\ell <\ka^3_\ep N(rst)^{3(1+\ep)}(X_fX_gX_h)^{3(1+\ep)}.
\end{align*}
Putting $X^3=X_fX_gX_h$, we obtain
\begin{align}\label{abc3}
X^{\ell-3(1+\ep)} <\ka_\ep N(rst)^{\frac{2}{3}+\ep}=\ka_\ep N(\frac{uvwA_fA_gA_h}{G^3})^{\frac{2}{3}+\ep}.
\end{align}

Again from \eqref{rst}, we have
\begin{align*}
rs(X_fX_h)^\ell \leq \left(\frac{rX^\ell_f+sX^\ell_h}{2}\right)^2=\frac{t^2X^{2\ell}_g}{4}
\end{align*}
implying
\begin{align*}
X_fX_hX_g \leq \left(\frac{t^2}{4rs}\right)^{\frac{1}{\ell}}X^3_g=\left(\frac{w^2A^2_g}{4uvA_fA_h}\right)^{\frac{1}{\ell}}X^3_g.
\end{align*}
Therefore we have from \eqref{abc1} that
\begin{align}\label{Xg<1}
X^\ell_g<\ka_\ep \frac{N(rst)^{1+\ep}X^{3+3\ep}_g}{t}
\left(\frac{t^2}{4rs}\right)^{\frac{1+\ep}{\ell}}=\ka_\ep \frac{N(rst)^{1+\ep}X^{3+3\ep}_g}
{(4rst)^{\frac{1+\ep}{\ell}}t^{1-\frac{3(1+\ep)}{\ell}}}
\end{align}
i.e.,
\begin{align}\label{Xg<2}
X^{\ell-3(1+\ep)}_g<\ka_\ep \frac{N(rst)^{(1+\ep)(1-\frac{1}{\ell})}}
{4^{\frac{1+\ep}{\ell}}t^{1-\frac{3(1+\ep)}{\ell}}}=\ka_\ep \frac{N(\frac{uvwA_fA_gA_h}{G^3})^{(1+\ep)(1-\frac{1}{\ell})}}
{4^{\frac{1+\ep}{\ell}}(\frac{wA_g}{G})^{1-\frac{3(1+\ep)}{\ell}}}.
\end{align}
Observe that
\begin{align*}
\frac{N(rst)^{(1+\ep)(1-\frac{1}{\ell})}}
{4^{\frac{1+\ep}{\ell}}t^{1-\frac{3(1+\ep)}{\ell}}}\leq
\frac{N(rs)^{(1+\ep)(1-\frac{1}{\ell})}N(t)^{\ep+\frac{2(1+\ep)}{\ell}}}
{4^{\frac{1+\ep}{\ell}}}.
\end{align*}
Hence we also have from \eqref{Xg<2} that
\begin{align}\label{Xg<3}
X^{\ell-3(1+\ep)}_g<\ka_\ep \frac{N(\frac{uvA_fA_h}{G^2})^{(1+\ep)(1-\frac{1}{\ell})}
N(\frac{wA_g}{G})^{\ep+\frac{2(1+\ep)}{\ell}}}{4^{\frac{1+\ep}{\ell}}}.
\end{align}

\begin{lem}\label{B0<B1}
Let $\ell\ge 11$. Let $S_0=\{A_0, A_1, \ldots , A_{k-1}\}=\{B_0, B_1, \ldots, B_{k-1}\}$
with $B_0\leq B_1\leq \ldots \leq B_{k-1}$. Then
\begin{align*}
B_0\leq B_1<B_2\ldots <B_{k-1}.
\end{align*}
In particular $|S_0|\ge k-1$.
\end{lem}

\begin{proof}
Suppose there exists $0\le f<g<h<k$ with $\{f, g, h\}=\{i_1, i_2, i_3\}$ and
\begin{align*}
A_{i_1}=A_{i_2}=A \ {\rm and} \ A_{i_3}\leq A.
\end{align*}
By \eqref{uvw} and \eqref{G}, we see that max$(A_f, A_g, A_h)\le G$ and therefore
$r\le u<k, s\leq v<k$ and $t\leq w<k$. Since $X_g>k$, we get from the first inequality of
\eqref{Xg<2} with $\ep=\frac{3}{4}, N_\ep=\ka_\ep=1$ that
\begin{align*}
k^{\ell-3(1+\ep)}<(rs)^{(1+\ep)(1-\frac{1}{\ell})}t^{\ep+\frac{2(1+\ep)}{\ell}}<k^{2+3\ep}
\end{align*}
implying $\ell <5+6\ep=5+\frac{9}{2}$. This is a contradiction since $\ell\ge 11$. Therefore
either $A_i$'s are distinct or if $A_i=A_j=A$, then $A_m>A$ for $m\notin \{i, j\}$ implying
the assertion.
\end{proof}

As a consequence, we have

\begin{coro}\label{d-odd}
Let $d$ be even and $\ell \ge 11$. Then $k\le 13$.
\end{coro}

\begin{proof}
Let $d$ be even and $\ell\ge 11$. Then we get from \eqref{erd} with $S=S_0$ that
\begin{align*}
\prod_{A_i\in S'}A_i\leq (k-1)!2^{{\rm ord}_2((k-1)!)}=\prod_{2i+1\le k-1}(2i+1).
\end{align*}
On the other hand, since gcd$(n, d)=1$, we see that all $A_i$'s are odd and
$|S'|\geq |S_0|-\pi(k)\ge k-1-\pi(k)$ by Lemma \ref{B0<B1}. Hence
\begin{align*}
\prod_{A_i\in S'}A_i\geq \prod^{k-1-\pi(k)}_{i=1}(2i-1).
\end{align*}
This is a contradiction since $2(k-1-\pi(k))>k-1$ for $k\ge 14$.
\end{proof}

\begin{lem}\label{k<89}
Let $\ell \ge 11$. Then $k<400$.
\end{lem}

\begin{proof}
Assume that $k\ge 400$. By Corollary \ref{d-odd}, we may suppose that $d$ is odd.
Further by Corollary \ref{3<4k}, there exists $f<g<h$ with $h-f\le 8$ and
max$(A_f, A_g, A_h)\leq 4k$. Since $n+(k-1)d>k^\ell$, we observe that $X_f>k, X_g>k, X_h>k$
implying $X>k$. First assume that $N=N(rstX_fX_gX_h)<e^{37.12}$. Then taking
$\ep=\frac{3}{4}, N_\ep=1$ in \eqref{abc0}, we get
$400^{11}\le k^{11}\leq tX^\ell_g<N^{1+\frac{3}{4}}\le e^{37.12(1+\frac{3}{4})}$ which is a
contradiction. Hence we may suppose that $N\ge e^{37.12}\ge N_{\frac{3}{4}}$.

Note that we have $u+v=w\leq h-f\le 8$. We observe that $uvw$ is even. If $A_fA_gA_h$ is odd,
then $h-f, g-f, h-g$ are all even implying $1\leq u, v, w\leq 4$ or $N(uvw)\leq 6$ giving
$N(uvwA_fA_gA_h)\leq 6A_fA_gA_h$. Again if $A_fA_gA_h$ is even, then
$N(uvwA_fA_gA_h)\leq N((uvw)')A_fA_gA_h \leq 35A_fA_gA_h$ where $(uvw)'$ is the
odd part of $uvw$ and $N((uvw)')\leq 35$. Observe that $N((uvw)')$ is obtained when
$w=7, u=2, v=5$ or $w=7, u=5, v=2$. Thus we always have
$N(uvwA_fA_gA_h)\leq 35A_fA_gA_h\le 35\cdot (4k)^3$ since max$(A_f, A_g, A_h)\le 4k$.
Therefore taking $\ep=\frac{3}{4}$ in \eqref{abc3}, we obtain using $\ell\ge 11$ and $X>k$ that
\begin{align*}
k^{11-3(1+\frac{3}{4})}<\frac{6}{5\sqrt{28\pi}}35^{\frac{2}{3}+\frac{3}{4}}
(4k)^{3(\frac{2}{3}+\frac{3}{4})}.
\end{align*}
This is a contradiction since $k\ge 400$.  Hence the assertion.
\end{proof}

\section{Proof of Theorem \ref{ABCErdos} for $4\leq k<400$}

We assume that $\ell \ge 11$. It follows from the result of Saradha and
Shorey \cite[Theorem 1]{Contr} that $d>10^{15}$. Hence we may suppose
that $d>10^{15}$ in this section.

\begin{lem}\label{rk}
Let $r_k=[k+1-\pi(k)-\frac{\sum_{i\le k}\log i}{15\log 10}]$ and
\begin{align*}
I(k)=\{i\in [1, k]: P(n+id)>k\}.
\end{align*}
Then $|I(k)|\ge r_k$.
\end{lem}

\begin{proof}
Suppose not. Then $|I(k)|\leq r_k-1$. Let
\begin{align*}
I'(k)=\{i\in [1, k]: P(n+id)\le k\}=\{i\in [1, k]: n+id=A_i\}.
\end{align*}
We have $A_i=n+id\ge (n+d)$ for $i\in I'(k)$. Let $S=\{A_i: i\in I'(k)\}$. Then
$|S|\geq k+1-r_k$. From \eqref{erd}, we get
\begin{align*}
(k-1)!\ge \prod_{A_i\in S'}A_i\ge (n+d)^{|S'|}>d^{k+1-r_k-\pi(k)}.
\end{align*}
Since $d>10^{15}$, we get
\begin{align*}
k+1-\pi(k)-\frac{\sum_{i\le k}\log i}{15\log 10}<r_k=[k+1-\pi(k)-\frac{\sum_{i\le k}\log i}{15\log 10}].
\end{align*}
This is a contradiction.
\end{proof}

Here are some values of $(k, r_k)$.
\begin{center}
\begin{tabular}{|c|c|c|c|c|c|c|c|c|} \hline
$k$ & $7$ & $11$ & $13$ & $17$ & $18$ & $28$ & $30$ & $36$ \\ \hline
$r_k$ & $3$ & $6$ & $7$ & $10$ & $10$ & $18$ & $18$ & $23$ \\ \hline
\end{tabular}
\end{center}

We give the strategy here. Let $I_k=[0, k-1]\cap \Z$ and $a_0, b_0, z_0$ be given.
Let obtain a subset $I_0\subseteq I_k$ with the following properties:
\begin{enumerate}
\item $|I_0|\ge z_0\ge 3$.
\item $P(A_i)\leq a_0$ for $i\in I_0$.
\item $I_0\subseteq [j_0, j_0+b_0-1]$ for some $j_0$.
\item $X_0=\max_{i\in I_0}\{X_i\}>k$ and let $i_0\in I_0$ be such that $X_0=X_{i_0}$.
\end{enumerate}
For any $i, j\in I_0$, taking $\{f, g, h\}=\{i, j, i_0\}$, let $N=N(rstX_fX_gX_h)$.
Observe that $X_0\ge p_{\pi(k)+1}$ and further for any $f, g, h\in I_0$, we have
$N(uvw)\leq \prod_{p\le b_0-1}p$ and $N(A_fA_gA_h)\le \prod_{p\le a_0}p$. We will always
take $\ep=\frac{3}{4}, N_\ep=1$ so that $\ka_\ep=1$ in \eqref{abc0} to \eqref{Xg<3}.

\noindent
{\bf Case I:} Suppose there exists $i, j\in I_0$ such that $X_i=X_j=1$.
Taking $\{f, g, h\}=\{i, j, i_0\}$ and $\ep=\frac{3}{4}$, we obtain from \eqref{abc1}
and $\ell\geq 11$ that
\begin{align}\label{CaseI}
p^{\frac{37}{7}}_{\pi(k)+1}\leq X^{\frac{\ell}{1+\frac{3}{4}}-1}_0<N(uvwA_fA_gA_h)\leq \prod_{p\le \max\{a_0, b_0-1\}}p.
\end{align}

\noindent
{\bf Case II:} There is at most one $i\in I_0$ such that $X_i=1$. Then
$|\{i\in I_0: X_i>k\}|\ge z_0-1$. We take $a_1, b_1, z_1$ and find a subset
$U_0\subset I_0$ with the following properties:

\begin{enumerate}
\item $|U_0|\ge z_1\ge 3$, $\frac{z_0}{2}\leq z_1\le z_0$.
\item $P(A_i)\leq a_1$ for $i\in U_0$.
\item $U_0\subseteq [i, i+b_1-1]$ for some $i$.
\end{enumerate}
Let $X_1=\max_{i\in U_0}\{X_i\}\ge p_{\pi(k)+z_1-1}$ and $i_1$ be such that $X_{i_1}=X_1$.
Taking $\{f, g, h\}=\{i, j, i_1\}$ for some $i, j\in U_0$ and $\ep=\frac{3}{4}$, we obtain
from \eqref{Xg<2} and $\ell\ge 11$ that
\begin{align}\label{CaseII}
p^{\frac{23}{7}}_{\pi(k)+z_1-1}\leq X^{\frac{\ell}{1+\frac{3}{4}}-3}_0<N(uvwA_fA_gA_h)\leq
\prod_{p\le \max\{a_1, b_1-1\}}p
\end{align}
since $\ell\ge 11$. One choice is $(U_0, a_1, b_1, z_1)=(I_0, a_0, b_0, z_0)$. We state the other choice.

Let $b'=\max(a_0, b_0-1)$. For each $\frac{b_0}{2}-1<p\leq b'-1$, we remove those $i\in I_0$ such
that $p|(n+id)$. There are at most $2(\pi(b'-1)-\pi(\frac{b_0}{2}-1))$ such $i$. Let
$I'_0$ be obtained from $I_0$ after deleting those $i$'s. Then
$|I'_0|\ge z_0-2(\pi(b'-1)-\pi(\frac{b_0}{2}-1))$. Let
\begin{align*}
U_1=I'_0\cap [j_0, j_0+\frac{b_0}{2}-1] \ \ {\rm  and} \ \ U_1=I'_0\cap [j_0+\frac{b_0}{2}, j_0+b_0-1].
\end{align*}
Let $U_0\in \{U_1, U_2\}$ for which $|U_i|=\max(|U_1|, |U_2|)$ and choose one of them if $|U_1|=|U_2|$.
Then $|U_0|\ge \lceil \frac{z_0}{2}\rceil -\pi(b'-1)+\pi(\frac{b_0}{2}-1)=z_1$. Further
$P(A_i)\leq \frac{b_0}{2}-1=a_1$ and $b_1=\frac{b_0}{2}$. Further
$X_1=\max_{i\in U_0}\{X_i\}\ge p_{\pi(k)+z_1-1}$. Our choice of $z_0, a_0, b_0$ will imply that
$z_1\ge 3$.

\section*{4.1. $k\in \{4, 5, 7, 11\}$} We take $I_0=U_0=I_k, a_i=b_i=z_i=k$ for $i\in \{0, 1\}$ and hence
$N(uvwA_fA_gA_h)\leq \prod_{p\le k}p$. And the assertion follows since both \eqref{CaseI} and \eqref{CaseII} are contradicted.

\section*{4.2. $k\in \{13, 17, 19, 23\}$} We take
$I_0=\{i\in [1, 11]: p\nmid (n+id) \ {\rm for} \ 13\leq p\le 23\}$. Then by $r_{11}=6$ and
Lemma \ref{rk} with $k=11$, we see that $|I_0|\ge z_0=11-4>11-r_{11}\ge 11- |I(11)|$. Therefore
there exist an $i\in I_0\cap I_{11}$ and hence $X_i>23$. We take $U_0=I_0$, $a_i=b_i=11, z_1=z_0$ for
$i\in \{0, 1\}$ and hence $N(uvwA_fA_gA_h)\leq \prod_{p\le 11}p$. And the assertion follows
since both \eqref{CaseI} and \eqref{CaseII} are contradicted.

\section*{4.3. $29\le k\le 47$} We take $I_0=\{i\in [1, 17]: p\nmid (n+id) \ {\rm for} \ 17\leq p\le k\}$.
Then by $r_{17}=10$ and Lemma \ref{rk} with $k=17$, we have
$|I_0|\ge z_0=17-(\pi(k)-\pi(13))=23-\pi(k)\ge 23-\pi(47)=8>17-r_{17}\ge 17-|I(17)|$ implying that
there exists $i\in I_0$ with $X_i>k$. We take $a_i=13, b_i=17, z_i=23-\pi(k)$ for $i\in \{0, 1\}$ and hence $N(uvwA_fA_gA_h)\leq \prod_{p\le 13}p$. And the assertion follows since both \eqref{CaseI} and
\eqref{CaseII} are contradicted.

\section*{4.4. $k\ge 53$}

Given $m$ and $q$ such that $mq<k$, we consider the $q$ intervals
\begin{align*}
I_j=[(j-1)m+1, jm]\cap \Z \ {\rm for} \ 1\le j\le q
\end{align*}
and let $I'=\displaystyle{\cup^q_{j=1}}I_j$ and $I^{"}=\{i\in I': m\le P(A_i)\le k\}$.
There is at most one $i\in I'$ such that $mq-1<P(A_i)\le k$ and for each $2\le j\le q$, there are at most
$j$ number of $i\in I'$ such that $\frac{mq-1}{j}<P(A_i)\le \frac{mq-1}{j-1}$. Therefore
\begin{align*}
|I^{"}|&\le \pi(k)-\pi(mq-1)+\sum^q_{j=2}j\left(\pi(\frac{mq-1}{j-1})-\pi(\frac{mq-1}{j})\right)\\
&=\pi(k)+\sum^{q-1}_{j=1}\pi(\frac{mq-1}{j})-q\pi(m-1)=:T(k, m, q).
\end{align*}
Hence there is at least one $j$ such that $|I_j\cap I^{"}|\le [\frac{T(k, m, q)}{q}]$. We will choose $q$
such that $[\frac{T(k, m, q)}{q}]<r_m$. Let $I_0=I_j\setminus I^{"}$ and let $j_0$ be such that
$I_0\subseteq [(j_0-1)m+1, j_0m]$. Then $p|(n+id)$ imply $p<m$ or $p>k$ whenever $i\in I_0$. Further
$|I_0|\ge z_0=m-[\frac{T(k, m, q}{q}]$. Since $[\frac{T(k, m, q)}{q}]<r_m$, we get from
Lemma \ref{rk} with $k=m$ and $n=(j_0-1)m$ that there is an $i\in I_0$ with $X_i>k$. Further
$P(A_i)<m$ for all $i\in I_0$. Here are the choices of $m$ and $q$.
\begin{center}
\begin{tabular}{|c|c|c|c|c|c|} \hline
$k$ & $53\leq k< 89$ & $89\leq k<179$ & $179\le k<239$ & $239\le k< 367$ & $367\le k<433$ \\ \hline
$(m, q)$ & $(17, 3)$ & $(28, 3)$ & $(36, 5)$ & $(36, 6)$ & $(36, 10)$ \\ \hline
\end{tabular}
\end{center}
We have $a_0=m-1, b_0=m$ and $z_0=m-[\frac{T(k, m, q)}{q}]$ and we check that $z_0\ge 3$.
The Subsection 4.3($29\le k\le 47$) is in fact obtained by considering $m=17, q=1$. Now we
consider Cases I and II and try to get contradiction in both \eqref{CaseI} and \eqref{CaseII}.
For these choices of $(m, q)$, we find that the Cases I are contradicted. Further taking
$U_0=I_0, a_1=a_0=m-1, b_1=b_0=m, z_1=z_0$, we find that Case II is also contradicted for
$53\leq k<89$. Thus the assertion follows in the case $53\leq k<89$. So, we consider
$k\ge 89$ and try to contradict Cases $II$. Recall that we have $X_i>k$ for all but at
most one $i\in I_0$. Write $I_0=U_1\cup U_2$ where $U_1=I_0\cap [(j_0-1)m+1, (j_0-1)m+\frac{m}{2}]$
and $U_2=I_0\cap [(j_0-1)m+\frac{m}{2}+1, j_0m]$. Let $U'_0=U_1$ or $U'_0=U_2$ according as
$|U_1|\ge \frac{z_0}{2}$ or $|U_2|\ge \frac{z_0}{2}$, respectively. Let
$U_0=\{i\in U'_0: p\nmid A_i \ {\rm for} \ \frac{m}{2}\le p<m\}$. Then
$|U_0|\ge z_1:=\frac{z_0}{2}-(\pi(m-1)-\pi(\frac{m}{2}))=\frac{m-[\frac{T(k, m, q)}{q}]}{2}-(\pi(m-1)-\pi(\frac{m}{2}))\ge 3$. Further $p|(n+id)$ with $i\in U_0$ imply
$p<\frac{m}{2}$ or $p>k$. Now we have Case II with $a_1=\frac{m}{2}-1, b_1=\frac{m}{2}$
and find that \eqref{CaseII} is contradicted. Hence the assertion.

\section{$\ell=7$}

Let $\ell=7$. Assume that $k\geq exp(13006.2)$. Taking $\al=3, \be =\frac{1}{15}+\frac{2}{9}$ in
Lemma \ref{S1<k/4}, we get
\begin{align*}
|S_1(3)|=\{i\in [0, k-1]: A_i\leq 3k\}|> k(\frac{1}{15}+\frac{2}{9}).
\end{align*}
For $i$'s such that $A_i\in S_1(3)$, we have $X_i>k$ and we arrange these $X_i$'s in increasing
order as $X_{i_1}<X_{i_2}<\ldots < $. Then $X_{i_j}\ge p_{\pi(k)+j}$. Consider the set
$J_0=\{i: X_i\ge p_{\pi(k)+[\frac{k}{15}]-2}\}$. We have
\begin{align*}
|J_0|>k(\frac{1}{15}+\frac{2}{9})-\frac{k}{15}+2\geq 2\left([\frac{k-1}{9}]+1\right).
\end{align*}
Hence there are $f, g, h\in J_0$, $f<g<h$ such that $h-f\leq 8$. Also $A_i\le 3k$ and
$X=(X_fX_gX_h)^{\frac{1}{3}}\ge p_{\pi(k)+[\frac{k}{15}]-2}$.

First assume that $N=N(rstX_fX_gX_g)\ge exp(63727)\ge N_{\frac{1}{3}}$. Observe that
$uvw\le 70$ since $2\le u+v=w\le 8$, obtained at $2+5=7$. Taking $\ep=\frac{1}{3}$,
we obtain from \eqref{abc3} and max$(A_f, A_g, A_h)\le 3k$ that
\begin{align*}
p^3_{\pi(k)+[\frac{k}{15}]-2}<\frac{5}{6\sqrt{2\pi \cdot 6458}}N(uvwA_fA_gA_h)\leq
\frac{5\cdot 70\cdot (3k)^3}{6\sqrt{12920\pi}}.
\end{align*}
Since $\pi(k)>2$ we have $\pi(k)+[\frac{k}{15}]-2>\frac{k}{15}$ and hence $p_{\pi(k)+[\frac{k}{15}]-2}>
\frac{k}{15}\log \frac{k}{15}$ by Lemma \ref{pnt} $(ii)$. Therefore
\begin{align*}
\left(\log \frac{k}{15}\right)^3<\frac{350\cdot (3\cdot 15)^3}{6\sqrt{12920\pi}} \ {\rm or}
\ k<15\cdot exp\left(45\cdot \left(\frac{350}{6\sqrt{12920\pi}}\right)^{\frac{1}{3}}\right)
\end{align*}
which is a contradiction since $k\ge exp(13006.2)$.

Therefore we have $N=N(rstX_fX_gX_h)<exp(63727)$. We may also assume that $N>exp(3895)$ otherwise
taking $\ep=\frac{3}{4}$ in \eqref{abc0}, we get $k^7<X^7_g<N^{1+\frac{3}{4}}\le exp(3895\cdot \frac{7}{4})$
or $k<exp(\frac{3895}{4})$ which is a contradiction. Now we take $\ep=\frac{5}{12}$ in \eqref{abc0} to get
 $k^7<X^7_g<N^{1+\frac{5}{12}}\le exp(64266\cdot \frac{17}{12})$ or $k<exp(13006.2)$. Hence the assertion.

\section{Nagell-Ljungrenn equation: Proof of Theorem \ref{ABCnalu}}

Let $x>1, y>1, n>2$ and $q>1$ be a non-exceptional solution of \eqref{nalu}.
It was proved by Ljunggren \cite{Lju} that there are no further solutions of \eqref{nalu}
when $q=2$. Thus we may suppose that $q\ge 3$. Further it has been proved that $4\nmid n$ by
Nagell \cite{Nag}, $3\nmid n$ by Ljunggren \cite{Lju} and $5\nmid n, 7\nmid n$ by
Bugeaud, Hanrot and Mignotte \cite{bhm}. Therefore $n\ge 11$.
From \eqref{nalu}, we get
\begin{align*}
1+(x-1)y^q=x^n.
\end{align*}
Then $y<x^{\frac{n}{q}}\le x^{\frac{n}{3}}$ since $q\ge 3$ implying $N=N(x(x-1)y)<x^2y<x^{2+\frac{n}{3}}$.
From \eqref{3/4} in Theorem \ref{abcexplicit}, we obtain
\begin{align*}
x^n<N^{\frac{7}{4}}<x^{\frac{7}{2}+\frac{7n}{12}} \ {\rm implying} \
n<\frac{7}{2}+\frac{7n}{12}.
\end{align*}
This gives $n\le 8$ which is a contradiction.

\section{Fermat-Catalan Equation}

We may assume that each of $p, q, r$  is either $4$ or an odd prime.
Let $[p, q, r]$ denote all permutations of ordered triple $(p, q, r)$. The
Fermat's Last Theorem $(p, p, p)$ was proved by Wiles \cite{wiles};
$[3, p, p], [4, p, p]$ for $p\ge 7$ by Darmon and Merel \cite{dargra} and
$[3, 5, 5], [4, 5, 5]$ by Poonen; $[4, 4, p]$ by Bennett, Ellenberg, Ng \cite{benn}.
The signatures $[3, 3, p]$ for $p\le 10^9$
was solved by Chen and Siksek \cite{chen}, $[3, 4, 5]$ by Siksek and Stoll \cite{siksto}
and $[3, 4, 7]$ by Poonen, Schefer and Stoll \cite{poo}.
Hence we may suppose $(p, q, r)$ is different from those values.

We may assume that $x>1, y>1, z>1$. Then
\begin{align*}
x<z^{\frac{r}{p}}, y<z^{\frac{r}{q}}.
\end{align*}
Given $\ep>0$, by Theorem \ref{abcexplicit}, we have
\begin{align}\label{fceep}
z^r<\begin{cases}
N^{\frac{7}{4}}_\ep & {\rm if} \ N(xyz)<N_\ep\\
N(xyz)^{1+\ep}\leq (xyz)^{1+\ep} & {\rm if} \ N(xyz)\ge N_\ep.
\end{cases}
\end{align}
In particular, taking $\ep=\frac{3}{4}$, we get
\begin{align*}
z^r<(xyz)^{\frac{7}{4}}<z^{\frac{7}{4}(1+\frac{r}{p}+\frac{r}{q})}
\end{align*}
implying
\begin{align}\label{4/7}
\frac{4}{7}<\frac{1}{p}+\frac{1}{q}+\frac{1}{r}.
\end{align}
Thus we need to consider $[3, 3, p]$ for $p>10^9$ and $(p, q, r)\in Q$.
Let $\ep=\frac{34}{71}$. First assume that $N(xyz)\ge N_\ep$. Then
$$z^r<(xyz)^{1+\ep}<z^{(1+\ep)(1+\frac{r}{p}+\frac{r}{q})}$$
implying
\begin{align*}
\frac{1}{p}+\frac{1}{q}+\frac{1}{r}>\frac{1}{1+\ep}=\frac{71}{105}=
\frac{1}{3}+\frac{1}{5}+\frac{1}{7}.
\end{align*}
Therefore we may suppose that $N(xyz)<N_{\frac{34}{71}}$.
Then from \eqref{fceep} that max$(x^p, y^q, z^r)<N^{\frac{7}{4}}_{\frac{34}{71}}\le
e^{1758.3353}$ implying $x, y, z, p, q, r$ are all bounded. This will imply
that $[3, 3, p]$ with $p>10^9$ does not have any solution. Hence the assertion.
$\hfill \qed$

\section{Goormaghtigh Equation}

Let $d=$gcd$(x, y)$. From \eqref{goor}, we have
\begin{align*}
x^{m-1}+\cdots +x=y^{n-1}+\cdots +y
\end{align*}
implying ord$_p(x)=$ord$_p(y)$ for all primes $p|d$. Further
\begin{align*}
\sum^{m-1}_{i=1}(x^i-y^i)=(x-y)\left\{1+\sum^{m-1}_{i=2}\frac{x^i-y^i}{x-y}\right\}=y^{n-1}+\cdots +y^m
\end{align*}
which is
\begin{align*}
1+\sum^{m-1}_{i=2}\frac{x^i-y^i}{x-y}=\frac{y^m}{x-y}\frac{y^{n-m}-1}{y-1}.
\end{align*}
We observe that $d$ is coprime to $\frac{y^{n-m}-1}{y-1}$ and also to the left hand side. Therefore
\begin{align*}
{\rm ord}_p(x-y)=m\cdot {\rm ord}_p(x)=m\cdot {\rm ord}_p(y)=m\cdot {\rm ord}_p(d)
\end{align*}
for every prime $p|d$. Let $d_2=$gcd$(y-1, x-1, x-y)$ and $d_3$ be given by
$x-y=d^md_2d_3$. We observe that $d_2d_3=1$ if $n=m+1$ and $d_2d_3|(y+1)$ if $n=m+2$.
We now rewrite \eqref{goor} as
\begin{align}\label{goo1}
\frac{(y-1)x^m}{d^md_2}+d_3=\frac{(x-1)y^n}{d^md_2}.
\end{align}
Let
\begin{align*}
N=N(\frac{x^my^n(x-1)(y-1)d_3}{d^{2m}d^2_2})\le N(xy(x-1)(y-1)d_3)
\le \frac{xy(x-1)(y-1)d_3}{2^\del dd_2}
\end{align*}
where $\del=0$ if $2|dd_2$ and $1$ otherwise. Recall that $d=$gcd$(x, y)$ and $d_2|(x-1)$.
Let $\ep<\frac{3}{4}$. We obtain from \eqref{goo1} and Theorem \ref{abcexplicit} and
$x-y=d^md_2d_3$ that
\begin{align}\label{2G4}
\max \{\frac{(y-1)x^md_3}{(x-y)}, \frac{(x-1)y^nd_3}{x-y}\}<\begin{cases}
N^{\frac{7}{4}}_{\ep} & {\rm if} \ N<N_\ep\\
N^{1+\ep} & {\rm if} \ N\ge N_{\ep}.
\end{cases}
\end{align}
Assume that $N\ge N_\ep$. Then we obtain using \eqref{2G4} that
\begin{align}
x^m<&x^{2+2\ep}y^{1+2\ep}(x-y)\frac{d^{\ep}_3}{(2^\del dd_2)^{1+\ep}}<x^{4+5\ep}\label{xm} \\
y^n<&x^{1+2\ep}y^{1+\ep}(y-1)^{1+\ep}(x-y)\frac{d^{\ep}_3}{(2^\del dd_2)^{1+\ep}} \label{yn}.
\end{align}
since $y<x$ and $d_3\le x-y<x$. We observe that from \eqref{goor} that $x^{m-1}<2y^{n-1}$ implying
$x<2^{\frac{1}{m-1}}y^{\frac{n-1}{m-1}}$. This together with \eqref{yn}, $d_3\le x-y<x$ and $2^\del dd_2\ge 2$
gives
\begin{align}\label{yn1}
y^n<&2^{\frac{2+3\ep}{m-1}-1-\ep}y^{2+2\ep+\frac{n-1}{m-1}(2+3\ep )}.
\end{align}
From \eqref{xm}, we obtain $m<4+5\ep$ and further from \eqref{yn1},
we get $n<2+2\ep+\frac{n-1}{m-1}(2+3\ep )$ if $m>3$.

Let $\ep=\frac{3}{4}$ and $N_\ep=1$. Then $m\le 7$ and further
$7\le n\le 17$ if $m=6$ and $n\in \{8, 9\}$ if $m=7$. Let $m=7, n=m+1=8$. Then
$d_2d_3=1$ and we get from the first inequality of \eqref{xm} and $y<x$ that
$x^m<x^{4+4\ep}=x^7$ implying $7=m<7$, a contradiction. Let $m=7, n=m+2=9$. Then
$d_2d_3\le y+1$ and we get from \eqref{yn} with $x<2^{\frac{1}{m-1}}y^{\frac{n-1}{m-1}}$,
$d_3(y-1)<y^2$ and $2^\del dd_2\ge 2$ that
$y^n<2^{\frac{2+2\ep}{m-1}-1-\ep}y^{2+3\ep+\frac{n-1}{m-1}(2+2\ep )}<y^9$ which is a
contradiction again. Let $m=6$ and $n\in \{11, 16\}$. From
Nesterenko and Shorey \cite{nest}, we get
$y\le 8, 15$ when $n=11, 16$, respectively. For $2\le y\le 15$ and
$y+1\le x\le (\frac{y^n-1}{y-1}))^{\frac{1}{m-1}}$, we check that \eqref{goor}
does not hold. Therefore $n\notin \{11, 16\}$ when $m=6$.
Hence we have the first assertion of Theorem \ref{ABCgoor}.

Now we take $\ep=\frac{1}{18}$. Since $m\le 7$ and
$G<x$, we get an explicit bound of $x, y, m, n$ from \eqref{2G4} if $N<N_{\frac{1}{18}}$,
implying Theorem \ref{ABCgoor} in that case. Thus we may suppose that
$N\ge N_{\frac{1}{18}}$. Then we obtain from \eqref{xm} with $\ep=\frac{1}{18}$ that
$m<4+5\ep$ implying $m\in \{3, 4\}$ and further from \eqref{yn1} that $n<5$ if $m=4$.
This is a contradiction for $m=4$ since $n>m$ and $n\in \Z$.

Let $m=3$. We rewrite \eqref{goor} as
\begin{align}\label{goo3}
(2x+1)^2=4(y^{n-1}+\cdots +y)+1
\end{align}
By \cite{nest}, we may assume that $n\neq 5$. Let $n=4$ and
denote by $f(y)$ the polynomial on the right hand side of \eqref{goo3}. Let
$f'(\alpha)=0$. Then $\alpha=\frac{-1\pm \sqrt{2}i}{3}$ and we check that
$f(\alpha)\neq 0$. Therefore the roots of $f$ are simple. Now we apply Baker \cite{bakhyp}
to conclude that $y$ and hence $x$ are bounded by  effectively computable absolute
constant. Let $n\ge 6$. Now we rewrite \eqref{goor} as
\begin{align}\label{gom3}
4y^n=(y-1)(2x+1)^2+(3y+1).
\end{align}
Let $G=$gcd$(4y^n, (y-1)(2x+1)^2, 3y+1)$. Then $G=4, 2, 1$ according as
$4|(y-1), 4|(y-3)$ and $2|y$, respectively and we get from \eqref{gom3} that
\begin{align}\label{gom4}
\frac{4}{G}y^n=\frac{y-1}{G}(2x+1)^2+\frac{3y+1}{G}.
\end{align}
Let
\begin{align*}
N=N(\frac{4y(y-1)(2x+1)(3y+1)}{G^3})\le \frac{y(y-1)(2x+1)(3y+1}{G}<
\frac{6xy^3}{G_1}.
\end{align*}
Let $\ep=\frac{1}{12}$. We obtain from Theorem \ref{abcexplicit} with
$\ep=\frac{1}{12}$ that
\begin{align}\label{2m3G1}
\frac{4y^n}{G}<\begin{cases}
N^{\frac{7}{4}}_{\frac{1}{12}} & {\rm if} \ N<N_{\frac{1}{12}}\\
N^{1+\frac{1}{12}} & {\rm if} \ N\ge N_{\frac{1}{12}}.
\end{cases}
\end{align}
If $N<N_{\frac{1}{12}}$, then $y^n<N^{\frac{7}{4}}_{\frac{1}{12}}$ implying the assertion
of Theorem \ref{ABCgoor}. Hence we may suppose that $N\ge N_{\frac{1}{12}}$ and further $y$
is sufficiently large. Then we have from $x^2<2y^{n-1}$ that
\begin{align*}
4y^n<(6\sqrt{2}y^{\frac{n+5}{2}})^{1+\frac{1}{12}}.
\end{align*}
Therefore
\begin{align*}
n-\frac{13(n+5)}{24}<\frac{\frac{13}{12}\log (6\sqrt{2})-\log 4}{\log y}<
\frac{1}{24}
\end{align*}
since $y$ is sufficiently large. This is not possible since $n\ge 6$. Hence the assertion
$\hfill \qed$

\section*{Acknowledgments}

The second author would like to thank ISI, New Delhi where this work was initiated
during his visit in July 2011.


\begin{thebibliography}{100}
\bibitem[BEN10]{benn} M.A. Bennett, J.S. Ellenberg and N.C. Ng,
\emph{The Diophantine equation $A^4+2^\del B^2=C^n$}, International Journal of
Number Theory, {\bf 6} (2010), no. 2, 311--338.
\bibitem[Bak94]{baker} A. Baker, \emph{Experiments on the abc-conjecture},
Publ. Math. Debrecen {\bf 65}(2004), 253--260.
\bibitem[Bak69]{bakhyp} A. Baker, \emph{Bounds for the solutions of the
hyperelliptic equation}, Proc. Cambridge Philos. Soc. {\bf 65} (1969),
439--444.
\bibitem[Beu04]{beu} F. Beukers, \emph{The Diophantine equation
$Ax^p+By^q=Cz^r$}, Lectures held at Institut Henri Poincare,
September 2004, $http://www.math.uu.nl/people/beukers/Fermatlectures.pdf$
Publ. Math. Debrecen {\bf 65}(2004), 253--260.
\bibitem[BHM02]{bhm} Y. Bugeaud, G. Hanrot and M. Mignotte,
\emph{Sur l'equation diophantienne $\frac{x^n-1}{x-1}=y^q$},
Proc. London Math. Soc. {\bf 84}(2002), 59--78.
\bibitem[BuMi02]{bumi} Y. Bugeaud and and M. Mignotte,
\emph{L' equation de Nagell-Ljunggren $\frac{x^n-1}{x-1}=y^q$},
Enseign.  Math. {\bf 48}(2002), 147--168.
\bibitem[BuMi07]{bumh} Y. Bugeaud and P. Mihailescu,
\emph{On the Nagell-Ljunggren equation $\frac{x^n-1}{x-1}=y^q$},
Math. Scand. {\bf 101}(2007), 177--183.
\bibitem[ChSi09]{chen} I. Chen and S. Siksek,
\emph{ Perfect powers expressible as sums of two cubes},
Journal of Algebra {\bf 322} (2009), 638--656.
\bibitem[Coh07]{cohen} H. Cohen,
\emph{Number Theory, Volume II: Analytic and Modern Tools},
GTM 240, Springer-Verlag, 2007.
\bibitem[DaGr95]{dargra} H. Darmon and A. Granville, \emph{On the equations
$z^m=F(x, y)$ and $Ax^p+By^q=cZ^r$}, Bull. London Math. Soc. {\bf 27} (1995),
513--543.
\bibitem[Dus99a]{duspk} P. Dusart, \emph{The $k^{th}$ prime is greater than
$k(\ln k +\ln\ln k -1)$ for $k\geq 2$}, Math. Comp. {\bf 68} (1999), 411--415
\bibitem[Dus99b]{dus1} P. Dusart, \emph{In\'egaliti\'es explicites pour
$\psi(X), \theta(X), \pi(X)$ et les nombres premiers}, C. R. Math. Rep. Acad.
Sci. Canada {\bf 21(1)}(1999), 53-59.
\bibitem[Elk91]{elk} N. Elkies, \emph{ABC implies Mordell}, Int. Math. Res. Not.
{\bf 7} (1991), [Ch 9].
\bibitem[ErSe75]{ersel} P. Erd\H{o}s and J. L. Selfridge, \emph{The product of consecutive
integers is never a power}, Illinois J. Math. {\bf 19} (1975), 292-301.
\bibitem[GrTu02]{gran} A. Granville and T. J. Tucker, \emph{It's as easy as abc},
Notices of the AMS, {\bf 49}(2002), 1224-31.
\bibitem[GyHaSa04]{GyHaSa} K. Gy\H{o}ry, L. Hajdu and , N. Saradha,
\emph{On the Diophantine equation $n(n+d)\cdots(n+(k-1)d)=by^{\ell}$},
Canad. Math. Bull. {\bf 47}(2004), 373--388.
\bibitem[Kra99]{kraus} A. Kraus, \emph{On the Equation $x^p+y^q=z^r$: A Survey},
Ramanujan Journal {\bf 3} (1999), 315--333.
\bibitem[Lai04]{Mthesis} S. Laishram, \emph{Topics in Diophantine equations}, M.Sc. Thesis,
TIFR/Mumbai University, 2004, online at
{\rm http$://$www.isid.ac.in$/$$\sim$shanta/MScThesis.pdf}.
\bibitem[LaSh04]{shanta2} S. Laishram and T. N. Shorey, \emph{Number of prime
divisors in a product of terms of an arithmetic progression}, Indag. Math.,
{\bf 15(4)} (2004), 505-521.
\bibitem[Lju43]{Lju} W. Ljunggren, \emph{Noen Setninger om ubestemte
likninger av formen} $(x^n-1)/(x-1) = y^q$, Norsk. Mat. Tidsskr.
1. Hefte {\bf 25}(1943), 17--20.
\bibitem[Nag20]{Nag} T. Nagell, \emph{Note sur $\ell'$ \'equation
ind\'etermi\'n\'ee} $(x^n-1)/ (x-1) =y^q$, Norsk. Mat. Tidsskr.
{\bf 2} (1920), 75-78.
\bibitem [NeSh98]{nest} Yu.V. Nesterenko and T. N. Shorey, \emph{On an equation of Goormaghtigh},
Acta Arith. {\bf 83} (1998), 381--389.
\bibitem[PSS07]{poo}B. Poonen, E.F. Schaefer and M. Stoll, \emph{Twists of $X(7)$ and
primitive solutions to $x^2+ y^3=z^7$}, Duke Math. J. {\bf 137} (2007), 103--158.
\bibitem[Rob55]{robb} H. Robbins, \emph{A remark on Stirling's formula},
Amer. Math. Monthly {\bf 62} (1955), 26-29.
\bibitem[Rob83]{robin} G. Robin, \emph{Estimation de la fonction de Tchebychef
$\theta$ sur le k-i$e$me nombre premier et grandes valeurs de la fonction
$\om(n)$ nombre de diviseurs premiers de $n$}, Acta Arith. {\bf 42} (1983),
367-–389.
\bibitem[Sar]{abcSara} N. Saradha, \emph{Applications of Explicit $abc$-Conjecture
on two Diophantine Equatiosn}, a preprint.
\bibitem[SaSh05]{Contr} N. Saradha and T. N. Shorey, \emph{Contributions
towards a conjecture of  Erdos on perfect powers in arithmetic progressions},
Compositio Math., {\bf 141} (2005), 541-560.
\bibitem[SaSh03]{Sar} N. Saradha and T. N. Shorey, \emph{Almost squares and
factorisations in consecutive integers}, Compositio Math. {\bf 138} (2003), 113-124.
\bibitem[Sho02a]{Camb} T.N. Shorey, \emph{Powers in arithmetic progression}, In: G.
W\"ustholz (ed.) A Panorama in Number Theory or The view from Baker's Garden,
Cambridge Univ. Press (2002), 325-336.
\bibitem[Sho02a]{hba} T. N. Shorey, \emph{An equation of Goormaghtigh and diophantine
approximations}, Current Trends in Number Theory, edited by
S.D.Adhikari, S.A.Katre and B.Ramakrishnan, Hindustan Book Agency, New Delhi
(2002), 185--197.
\bibitem[Sho02b]{rims} T.N. Shorey, \emph{Powers in arithmetic progression} (II),
Analytic Number Theory, RIMS Kokyuroku (2002), Kyoto University.
\bibitem[Sho99]{bam} T. N. Shorey, \emph{Exponential diophantine equations involving
products of consecutive integers and related equations}, Number Theory ed.
R.P. Bambah, V.C. Dumir and R.J. Hans-Gill, Hindustan Book Agency (1999),
463-495.
\bibitem[ShTi90]{shoti} T. N. Shorey and R. Tijdeman,
\emph{Perfect powers in products of terms in an arithmetical progression},
Compositio Math. {\bf 75} (1990), 307-344.
\bibitem[SiSt90]{siksto} S. Siksek and M. Stoll, \emph{Partial descent on
hyperelliptic curves and the generalized Fermat equation $x^3+y^4+z65=0$},
Bulletin of the LMS, in press.
\bibitem[TaWi95]{tay} R. Taylor and A. Wiles, \emph{Ring-theoretic properties
of certain Hecke algebras}, Annals of Mathematics {\bf 141} (1995), 553--572.
\bibitem[Wil95]{wiles} A. Wiles, \emph{Modular elliptic curves and Fermat's Last
Theorem}, Annals of Mathematics {\bf 141} (1995), 443--551.
\bibitem[ABC3]{triple} \emph{ABC triples}, page maintained by Bart de Smit at
$http://www.math.leidenuniv.nl/\sim desmit/abc/index.php?sort=1$, see also
$http://rekenmeemetabc.nl/Synthese\_resultaten$, $http://www.math.unicaen.fr/\sim nitaj/abc.html$.

\end{thebibliography}
\end{document}